\documentclass[11pt,a4paper]{article}

\usepackage{color}
\usepackage{graphics,epsfig}
\usepackage{amsfonts,amssymb,amsmath,amsthm}

\newtheorem{lem}{Lemma}[section]
\newtheorem{cor}[lem]{Corollary}
\newtheorem{thm}[lem]{Theorem}
\newtheorem{assumption}[lem]{Assumption}

\theoremstyle{definition}

\theoremstyle{remark}

\numberwithin{equation}{section}

\newcommand{\ep}{\varepsilon}
\newcommand{\ue}{u^\ep}

\newcommand{\eun}{\displaystyle{\frac{1}{\ep}}}
\newcommand{\R}{\mathbb{R}}
\newcommand{\vsp}{\vspace{8pt}}
\newcommand{\di}{\displaystyle}
\newcommand{\Pe}{(P^{\;\!\ep})}
\newcommand{\Pz}{(P^{\;\!*})}

\newcommand{\pt}{\ep|\ln\ep|m_1e^{m_2t}}

\title{Sharp interface limit of the Fisher-KPP equation}
\author{ }
\date{}

\begin{document}

\maketitle \vspace{-20 mm}

\begin{center}

{\large\bf Matthieu Alfaro }\\[1ex]
I3M, Universit\'e de Montpellier 2,\\
CC051, Place Eug\`ene Bataillon, 34095 Montpellier Cedex 5, France,\\[2ex]

{\large\bf Arnaud Ducrot }\\[1ex]
UMR CNRS 5251 IMB and INRIA Sud-Ouest ANUBIS, \\
Universit\'e de Bordeaux, 3, Place de la Victoire, 33000 Bordeaux France. \\[2ex]

\end{center}

\vspace{15pt}


\begin{abstract}

We investigate the singular limit, as $\ep \to 0$, of the Fisher
equation $\partial _t u=\ep \Delta u + \ep ^{-1}u(1-u)$ in the
whole space. We consider initial data with compact support plus,
possibly, perturbations very small as $\Vert x \Vert \to \infty$.
By proving both generation and motion of interface properties, we
show that the sharp interface limit moves by a constant speed,
which is the minimal speed of some related one-dimensional
travelling waves. We  obtain an estimate of the thickness of the
transition layers. We also exhibit initial data \lq\lq not so
small" at infinity
which do not allow the interface phenomena.\\

\end{abstract}

\section{Introduction}\label{s:intro}

Reaction diffusion equations with logistic nonlinearity were
introduced in the pioneer works of Fisher \cite{Fish} or
Kolmogorov, Petrovsky and Piskunov \cite{Kol-Pet-Pis}. The
equations read as
\begin{equation}\label{KPP}
\frac{\partial u}{\partial t}(t,x)=\Delta
u(t,x)+u(t,x)(1-u(t,x))\,,\;\;t>0\,,\;\;x\in\R^N\,,
\end{equation}
supplemented together with some suitable initial conditions. This
kind of equation was widely used in the literature to model
phenomena arising in population genetics, \cite{Fish,AW}, or in
biological invasions, \cite{Shi-Kaw, PM, MPV} and the references
therein. We also refer to \cite{BHL} for some extensions to
biological invasion in nonhomogeneous media with logistic
dynamics. The main property of equation \eqref{KPP} is to admit
(biologically relevant) travelling wave solutions with some
semi-infinite interval of admissible wave speed. The aim of this
work is to focus on the ability of equation \eqref{KPP} to
generate some interfaces and to propagate them. Such a property is
strongly related to these wave solutions. In order to observe such
a property, we shall rescale equation \eqref{KPP} by putting
\begin{equation*}
u^\ep(t,x)=u\left(\frac{t}{\ep},\frac{x}{\ep}\right)\,.
\end{equation*}
Therefore, we focus on the singular limit problem
\[
 \Pe \quad\begin{cases}
 \partial _t u= \ep \Delta u+\eun u(1-u)&\text{in }(0,\infty)\times \R ^N  \vspace{3pt}\\
 u(0,x)=u_{0,\ep}(x) &\text{in }\R ^N\,,
 \end{cases}
\]
as the parameter $\ep>0$, related to the thickness of a diffuse
interfacial layer, tends to zero.

We shall assume the following properties on the initial data.
\begin{assumption}\label{H1}
We assume that $u_{0,\ep}=g+h_\ep$ where
\begin{itemize}
\item [(i)] $g$ is a smooth (at least of the class $C^2$),
nonnegative and compactly supported function. We define $
\Omega_0:={\rm supp}\; g$. \item [(ii)] $0\in \Omega_0$. \item
[(iii)] $h_\ep$ is a smooth and nonnegative function and there
exist $\lambda\geq 1$ and $M>0$ such that, for all $\ep >0$ small
enough,
\begin{equation*}
h_\ep(x)\leq M e^{- \lambda\frac{\|x\|}{\ep}}\,,\;\;\forall
x\in\R^N\,.
\end{equation*}
\end{itemize}
\end{assumption}

\begin{assumption}\label{H2}
We assume that $\Omega_0$ is convex.
\end{assumption}

\begin{assumption}\label{H3}
We assume the existence of $\delta >0$ such that, if $n$ denotes
the Euclidian unit normal vector exterior to the ``initial
interface" $\Gamma _0:=\partial \Omega_0$, then
\begin{equation}\label{pente}
\left| \frac{\partial g}{\partial n}(y)\right|\geq \delta
\quad\text{ for all } y\in \Gamma _0\,.
\end{equation}
\end{assumption}

\vsp \noindent {\bf Heuristics.} In view of Assumption \ref{H1},
it is  standard that, for each $\ep >0$, Problem $\Pe$ has a
unique smooth solution $\ue(t,x)$ on $[0,\infty) \times \R ^N$. As
$\ep \rightarrow 0$, a formal asymptotic analysis shows the
following: in the very early stage, the diffusion term $\ep \Delta
\ue$ is negligible compared with the reaction term $\ep
^{-1}\ue(1-\ue)$ so that, in the rescaled time scale $\tau=t
/\ep$, the equation is well approximated by the ordinary
differential equation $\partial _\tau \ue =\ue(1-\ue)$. Hence the
value of $\ue$ quickly becomes close to either $1$ or $0$ in most
part of $\R^N$, creating a steep interface (transition layer)
between the regions $\{\ue\approx 1\}$ and $\{\ue\approx 0\}$.
Once such an interface develops, the diffusion term becomes large
near the interface, and comes to balance with the reaction term.
As a result, the interface ceases rapid development and starts to
propagate in a much slower time scale.

\vsp \noindent {\bf The limit free boundary problem.} To study
this interfacial behavior, we consider the asymptotic limit of
$\Pe$ as $\ep\rightarrow 0$. Then the limit solution $\tilde u
(t,x)$ will be a step function taking the value $1$ on one side of
the moving interface, and $0$ on the other side. This sharp
interface, which we will denote by $\Gamma ^*_t$, obeys the law of
motion
\[
 \Pz\quad\begin{cases}
 \, V_{n}=c^*
 \quad \text { on } \Gamma ^*_t \vspace{3pt}\\
 \, \Gamma_t^*\big|_{t=0}=\Gamma_0\,,
\end{cases}
\]
where $V_n$ denotes the normal velocity of $\Gamma ^* _t$ in the
exterior direction and $c^*$ the minimal speed of some related
one-dimensional travelling waves (see subsection \ref{ss:tw} for
details).

Since the region enclosed by $\Gamma _0$, namely $\Omega _0$, is
smooth and convex, Problem $\Pz$, possesses a unique smooth
solution on $[0,\infty)$, which we denote by $\Gamma ^*=\bigcup
_{t\geq 0} (\{t\}\times\Gamma^*_t)$. Hereafter, we fix $T>0$ and
work on $(0,T]$.

For each $t\in (0,T]$, we denote by $\Omega ^* _ t$  the region
enclosed by the hypersurface $\Gamma ^* _t$. We define a step
function $\tilde u (t,x)$ by
\begin{equation}\label{u}
\tilde u  (t,x)=\begin{cases}
\, 1 &\text{ in } \Omega ^* _t\\
\, 0 &\text{ in } \R ^N \setminus \overline{\Omega ^* _t}
\end{cases} \quad\text{for } t\in(0,T]\,,
\end{equation}
which represents the asymptotic limit of $\ue$ (or the {\it sharp
interface limit}) as $\ep\to 0$.

\vsp \noindent {\bf Known related results.} The question of the
convergence of Problem $\Pe$ to $\Pz$ has been addressed when the
initial data $u_{0,\ep}$ does not depend on $\ep$ and is compactly
supported : first by Freidlin \cite{Frie} using probabilistic
methods and later by Evans and Souganidis \cite{Ev-Soug} using
Hamilton Jacobi technics (in this framework we also refer to
\cite{Bar-Eva-Sou, Bar-Soug2}). The purpose of the present work is
to provide a new proof of convergence for Problem $\Pe$ by using
specific reaction-diffusion tools such as the comparison
principle. These technics were recently used by Hilhorst {\it et
al.} in \cite{HKLM} to consider the generation and propagation of
interfaces for a degenerated Fisher equation. Degenerated Fisher
equation have some semi-compactly supported travelling wave
solutions, which is essential in the proof given in \cite{HKLM}.
However Equation \eqref{KPP} does not possess solution with such a
property. We adapt these technics to study the singular limit of
$\Pe$. For bistable nonlinearities, we refer to \cite{A-Hil-Mat}
where an optimal estimate of the transition layers is provided for
the Allen-Cahn equation whose singular limit is a motion by mean
curvature.

Let us precise that we improve the convergence of $\Pe$ in two
directions. On the one hand, we extend the set of initial data for
which Problem $\Pe$ has a singular limit, by allowing positive
initial data with a suitable behavior at infinity. Moreover, we
also exhibit initial data \lq\lq not so small" at infinity which
do not allow the interface phenomena. On the other hand we provide
an $\mathcal O (\ep|\ln \ep|)$ estimate of the thickness of the
transition layers of the solutions $\ue$.

\vsp \noindent {\bf Results.} Our main result, Theorem
\ref{th:results}, describes the profile of the solution after a
very short initial period. It asserts that: given a virtually
arbitrary initial data, the solution $\ue$ quickly becomes close
to 1 or 0, except in a small neighborhood of the initial interface
$\Gamma _0$, creating a steep transition layer around $\Gamma _0$
({\it generation of interface}). The time needed to develop such a
transition layer, which we will denote by $t ^\ep$, is of order
$\ep  |\ln\ep|$. The theorem then states that the solution $\ue$
remains close to the step function $\tilde u$ on the time interval
$[t^\ep,T]$ ({\it motion of interface}); in other words, the
motion of the transition layer is well approximated by the limit
interface equation $\Pz$.

\begin{thm}[Generation, motion and thickness of interface]\label{th:results}
Let Assumptions \ref{H1}, \ref{H2} and \ref{H3} be satisfied. Then
there exist positive constants $\alpha$ and $\mathcal C$ such
that, for all $\ep >0$ small enough and for all $t^\ep \leq t \leq
T$, where
$$
t^\ep:=\alpha \ep  |\ln \ep|\,,
$$
we have
\begin{equation}\label{resultat}
\ue(t,x) \in
\begin{cases}
\,[0,1+\ep]\quad&\text{if}\quad
x\in\mathcal N_{\mathcal C\ep|\ln\ep|}(\Gamma ^*_t)\\
\,[1-2\ep,1+\ep]\quad&\text{if}\quad x\in\Omega ^* _
t\setminus\mathcal N_{\mathcal C\ep|\ln\ep|}(\Gamma ^*
_t)\\
\,[0,\ep]\quad&\text{if}\quad x\in (\R^N\setminus \overline{\Omega
^*_t})\setminus\mathcal N_{\mathcal C\ep|\ln\ep|}(\Gamma ^* _t)\,,
\end{cases}
\end{equation}
where $\mathcal N _r(\Gamma ^* _t):=\{x\in \R ^N:\;\;
dist(x,\Gamma ^* _t)<r\}$ denotes the tubular $r$-neighborhood of
$\Gamma ^*_t$.
\end{thm}

As a immediate consequence of the above Theorem, we collect the
convergence result.

\begin{cor}[Convergence]\label{cor:cv}Let Assumptions \ref{H1}, \ref{H2}
and \ref{H3} be satisfied. Then, as $\ep\to 0$, $\ue$ converges to
$\tilde u $ everywhere in $\bigcup _{0< t\leq T}(\{ t\}\times
\Omega ^* _t)$ and $\bigcup _{0< t \leq T}\left(\{ t\}\times (\R^N
\setminus \overline{\Omega ^* _t})\right)$.
\end{cor}

Next, it is intuitively clear that the limit problem depends
dramatically on the initial data. In order to underline this fact,
we show that for initial data \lq\lq not so small" at infinity,
the solution $\ue$ tends to 1 everywhere, as $\ep \to 0$.
Therefore, the interface phenomena does not occur in this case.

\begin{assumption}\label{ass6}
We assume that there exist $n>0$, $m>0$, $M>0$ and $C_0>0$ such
that, for all $\ep >0$ small enough,
\begin{equation*}
 \frac{m}{1+\di\frac{\| x\|
^n}{\ep^n}}\leq   u_{0,\ep}(x)\leq M\,,\;\;\forall x\in\R^N\,,
\end{equation*}
\begin{equation*}
\|u_{0,\ep}\|_\infty +\|\nabla u_{0,\ep}\|_\infty+\|\Delta
u_{0,\ep}\|_\infty\leq C_0\,.
\end{equation*}
\end{assumption}

\begin{thm}[There is no interface]\label{THEO6.1}
Let Assumption \ref{ass6} be satisfied. Then
\begin{equation*}
\lim_{\ep\to 0^+} u^\ep(t,x)=1\,,\;\;\forall (t,x)\in
(0,\infty)\times\R^N\,.
\end{equation*}
\end{thm}

\vsp \noindent {\bf Plan.} The organization of this paper is as
follows. In Section \ref{s:materials}, we present the basic tools
that will be used in later sections for the construction of sub-
and super-solutions. In Section \ref{s:generation} we construct
sub-solutions for very small times and super-solutions for all
times. They enable to prove a generation of interface property.
Section \ref{s:motion} is devoted to the construction of
sub-solutions for all times; their role is to control the solution
$\ue$ from below, while the sharp interface limit propagates. In
Section \ref{s:proof}, by using our different sub- and
super-solutions we prove the main result, Theorem
\ref{th:results}. Last, in Section \ref{s:small} we prove Theorem
\ref{THEO6.1}.

\section{Materials}\label{s:materials}

Let us recall that, in the classical works of Fisher \cite{Fish}
and Kolmogorov, Petrovsky and Piskunov \cite{Kol-Pet-Pis}, the
authors consider a monostable nonlinearity $f$ which is smooth and
such that $f(u)>0$ if $u\in(0,1)$, $f(u)<0$ if
$u\in(-\infty,0)\cup(1,\infty)$ and $f'(0)>0$. For the sake of
simplicity we select $f(u)=u(1-u)$ through this work.

\subsection{A monostable ODE}\label{ss:ode}

The generation of interface is initiated by the dynamics of the
corresponding ordinary differential equation. Therefore, we gather
here well-known facts about the logistic dynamics that will be
extensively used in the sequel.

To be more precise, the generation of interface is strongly
related to the dynamical properties of the non-rescaled
corresponding ordinary differential equation associated to $\Pe$,
that is
\begin{equation*}
\frac{dz(t)}{dt}=z(t)(1-z(t))\,,\;\;t>0\,.
\end{equation*}
In the sequel, for technical reasons we shall apply the semiflow
generated by the above dynamical system to negative initial data.
In order to have some good dynamical properties, let us modify the
monostable nonlinearity $u\to u(1-u)$  on $(-\infty,0)$ so that
the modified function, we call it $\bar f:\R\to\R$, is of the
class $C^2$ and enjoys the bistable assumptions. More precisely,
$\bar f$ has exactly three zeros $-1<0<1$ and
\begin{equation}
{\bar f}'(-1)<0\,, \qquad \bar f'(0)=1>0\,, \qquad \bar
f'(1)=-1<0\,.
\end{equation}
Note that $\bar f(u)=u(1-u)$ if $u\geq 0$. As done in Chen
\cite{C1}, we consider $\bar f _\ep$ a slight modification of
$\bar f$ defined by
$$
\bar f
_\ep(u):=\psi(u)\frac{u-\ep|\ln\ep|}{|\ln\ep|}+(1-\psi(u))\bar
f(u)\,,
$$
with $\psi$ a smooth cut-off function satisfying conditions
(29)---(32) as they appear in \cite{HKLM}. As explained in
\cite{HKLM},
\begin{equation}\label{fep}
\bar f _ \ep (u) \leq \bar f(u)\quad \text{ for all }u\in\R\,.
\end{equation}

Then we defined $w(s,\xi)$ as the semiflow generated by the
ordinary differential equation
\begin{equation}\label{ode}
\begin{cases}
\di{\frac{dw}{ds}}(s,\xi)=\bar f_ \ep (w(s,\xi))\,,\;\;s>0\,,\vsp\\
w(0,\xi)=\xi\,.
\end{cases}
\end{equation}
Here $\xi$ ranges over the interval
$[-\|g\|_{\infty}-M-1,\|g\|_{\infty}+M+1]$. We claim that
$w(s,\xi)$ has the following properties (for proofs, see \cite{C1}
or \cite{HKLM}).

\begin{lem}[Behavior of $w$]\label{LE2}The following
holds for all $\xi \in[-\|g\|_{\infty}-M-1,\|g\|_{\infty}+M+1]$.
\begin{itemize}
\item [(i)] $\text{ If}\quad \xi \geq \ep|\ln \ep| \quad\text{ then}\quad w(s,\xi)\geq\ep|\ln\ep|>0 \quad\text{ for all } s>0\,.$\\
$\text{ If}\quad \xi <0 \quad\text{ then}\quad w(s,\xi)<0
\quad\text{ for all } s>0\,.$\\
$\text{ If}\quad \xi \in(0,\ep|\ln\ep|) \quad\text{ then}\quad
w(s,\xi)>0 \quad\text{ for all } s\in(0,s_\ep(\xi)), \text{ with
}$
$$
s_\ep(\xi):=|\ln\ep|\left|\ln\left(1-\frac{\xi}{\ep|\ln\ep}\right)\right|\,.
$$
 \item [(ii)] $w(s,\xi) \in
(-\|g\|_{\infty}-M-1,\|g\|_{\infty}+M+1)\quad\text{ for all }
s>0\,.$ \item [(iii)] $w$ is of the class $C^2$ with respect to
$\xi$ and $$ w_\xi (s,\xi)>0\quad\text{ for all } s>0\,.$$

 \item [(iv)]For all $a>0$, there exists a constant $C(a)$ such that
$$
 \left|\di{\frac{w_{\xi\xi}}{w_\xi}(s,\xi)}\right|\leq \frac {C(a)} \ep \quad\text{ for all }
 0<s\leq a|\ln \ep|\,.
 $$
\item [(v)] There exists a positive constant $\alpha$ such that,
for all $s\geq \alpha \ep|\ln \ep|$, we have
$$
\text{ if }\quad \xi \in[\ep|\ln\ep|,\|g\|_{\infty}+M+1]
\quad\text{ then }\quad 0<w(s,\xi)\leq 1+\ep\,,
$$
and
$$
\text{ if }\quad \xi \in[3\ep|\ln\ep|,\|g\|_{\infty}+M+1]
\quad\text{ then }\quad 1-\ep \leq w(s,\xi)\,.
$$
\end{itemize}
\end{lem}

\subsection{Travelling waves}\label{ss:tw}

For the self-containedness of the present paper, we recall here
well-known facts concerning one dimensional travelling waves
related to our problem. We refer the reader to
\cite{Aro-Wei1,Volpert-Volpert-Volpert} and the references
therein.

\vsp A travelling wave is a couple $(c,U)$ with $c>0$ and $U\in
C^2(\R,\R)$ a function such that
\begin{equation}\label{tw}
\begin{cases}
 {U}''(z)+cU'(z)+U(z)(1-U(z))=0\quad \text{ for all } z\in \R\\
 U(-\infty)=1\\
 U(\infty)=0\,.
\end{cases}
\end{equation}
Define $c^*:=2$. Then the following holds.
\begin{enumerate}
\item [(i)] For all $c\geq c^*$ there exists a unique (up to a
translation in $z$) travelling wave denoted by $(c,U)$ or
$(c^*,U^*)$. It is positive and monotone. \item [(ii)] For all
$0<c<c^*$, there exists a unique (up to a translation in $z$) and
non monotone travelling wave $(c,U)$. It changes sign. In the
sequel, for each $c\in (0,c^*)$ we select $U$ as the solution of
\begin{equation}\label{tw-sign}
\begin{cases}
{U}''(z)+cU'(z)+U(z)(1-U(z))=0\quad \text{ for all } z\in \R\\
 U(-\infty)=1\\
 U(\infty)=0\\
 U(z)>0 \quad\text{ for all } z<0\\
 U(0)=0\,.
 \end{cases}
\end{equation}
\end{enumerate}

\begin{lem}[Behavior of $U$] Let $c>0$ be arbitrary and consider the associated
travelling wave $(c,U)$. Then there exist constants $C>0$ and $\mu
>0$ such that
\begin{equation}\label{moins-infini}
0 <1-U(z)\leq Ce^{-\mu|z|} \quad \text{ for } z\leq 0\,,
\end{equation}
\begin{equation}
|U(z)|\leq Ce^{-\mu|z|} \quad \text{ for } z\geq 0\,,
\end{equation}
\begin{equation}\label{estimate-tw}
|U'(z)|+|U''(z)|\leq Ce^{-\mu |z|} \quad \text{ for all } z\in
\R\,.
\end{equation}
\end{lem}
Note that, as easily seen from the standard proof, the constants
$C>0$ and $\mu >0$ depend continuously on $c$. This fact shall be
used in Lemma \ref{super-sub-motion}.

At last, since a precise behavior of the wave with the minimal
speed shall be necessary, let us recall that the travelling wave
solution $U^*$ associated to $c=c^*$ satisfies
\begin{equation}\label{behaviour}
\gamma^- ze^{-z}\leq U^*(z)\leq \gamma^+ ze^{-z}\,,\;\;\forall z
\geq 1\,,
\end{equation}
for two constants $0<\gamma ^-<\gamma ^+$.

\subsection{Cut-off signed distance functions}\label{ss:distance}

For $c>0$ we denote by $\Gamma ^c=\bigcup _{t\geq 0} (\{t\} \times
\Gamma^c_t)$ the smooth solution of the free boundary problem
\[
 (P^c)\quad\begin{cases}
 \, V_{n}=c
 \quad \text { on } \Gamma ^c _t \vspace{3pt}\\
 \, \Gamma ^c _t\big|_{t=0}=\Gamma_0\,.
\end{cases}
\]
If $c=c^*$, we naturally use the notations $\Gamma ^*=\bigcup
_{t\geq 0} (\{t\}\times \Gamma^*_t)$ and $(P^*)$. Note that since
the region enclosed by $\Gamma _0$, namely $\Omega _0$, is convex,
these solutions do exist for all $t\geq 0$. For each $t\geq 0$, we
denote by $\Omega ^c_t$ the region enclosed by the hypersurface
$\Gamma ^c_t$.

Let $\widetilde d$ be the signed distance function to $\Gamma ^c$
defined by
\begin{equation}\label{eq:dist}
\widetilde d (t,x)=
\begin{cases}
-&\hspace{-10pt}\mbox{dist}(x,\Gamma ^c_t)\quad\text{ for }x\in\Omega ^c _t \\
&\hspace{-10pt} \mbox{dist}(x,\Gamma ^c_t) \quad \text{ for }
x\in\R ^N \setminus \Omega ^c _t\,,
\end{cases}
\end{equation}
where $\mbox{dist}(x,\Gamma ^c_t)$ is the distance from $x$ to the
hypersurface  $\Gamma ^c _t$. We remark that $\widetilde d=0$ on
$\Gamma ^c$ and that $|\nabla \widetilde d|=1$ in a neighborhood
of $\Gamma ^c$.

We now introduce the ``cut-off signed distance function" $d$,
which is defined as follows. Recall that $T>0$ is fixed. First,
choose $d_0>0$ small enough so that $\widetilde d$ is smooth in
the tubular neighborhood of $\Gamma ^c$
\[
 \{(t,x) \in [0,T] \times \R ^N:\;\;|\widetilde{d}(t,x)|<3
 d_0\}\,.
\]
Next let $\zeta(s)$ be a smooth increasing function on $\R$ such
that
\[
 \zeta(s)= \left\{\begin{array}{ll}
 s &\textrm{ if }\ |s| \leq d_0\vspace{4pt}\\
 -2d_0 &\textrm{ if } \ s \leq -2d_0\vspace{4pt}\\
 2d_0 &\textrm{ if } \ s \geq 2d_0\,.
 \end{array}\right.
\]
We then define the cut-off signed distance function $d$ by
\begin{equation}
d(t,x)=\zeta\left(\tilde{d}(t,x)\right)\,.
\end{equation}
If $c=c^*$, we naturally use the notation $d^*$.

Note that
\begin{equation}\label{norme-un}
\text{ if } \quad |d(t,x)|< d_0 \quad \text{ then }\quad |\nabla
d(t,x)|=1\,,
\end{equation}
and that the equation of motion $(P^c)$ is now written as
\begin{equation}\label{interface}
 \partial _t d(t,x)+c=0 \quad\
 \textrm{ on}\ \; \Gamma ^c_t=\{x \in \R ^N:\;\;
d(t,x)=0\}\,.
\end{equation}
Then the mean value theorem provides a constant $N>0$ such that
\begin{equation}\label{MVT}
|\partial _t d(t,x)+c|\leq N|d(t,x)| \quad \textrm{ for all }
(t,x) \in [0,T]\times \R ^N\,.
\end{equation}
Moreover, there exists a constant $C>0$ such that
\begin{equation}\label{est-dist}
|\nabla d (t,x)|+|\Delta d (t,x)|\leq C\quad \textrm{ for all }
(t,x) \in [0,T]\times \R ^N\,.
\end{equation}

At least, note that the constants $d_0>0$, $C>0$ and $N>0$ depend
continuously on $c$. This fact shall be used in Lemma
\ref{super-sub-motion}.

\section{Generation of interface}\label{s:generation}

The aim of this section is to prove a generation of interface
property for Problem $\Pe$. We shall prove that after a small time
$t^\ep$ of order $\ep|\ln\ep|$ the solution $u^\ep$ becomes well
prepared and looks like a rescaled wave solution. We shall more
precisely prove the following result.
\begin{thm}[Generation of interface]\label{th:gene}
Let Assumption \ref{H1} be satisfied. Then there exist $k>0$,
$\alpha>0$ such that, for all $\ep>0$ small enough, the following
holds.
\begin{itemize}
\item [(i)] For all $x\in \Omega_0$ such that $g(x)\geq
k\ep|\ln\ep|$ we have
 \begin{equation*}
 u^\ep(t^\ep,x)\geq 1-\ep\,,
 \end{equation*}
 wherein $t^\ep:=\alpha\ep |\ln\ep|$.

 \item [(ii)] For all $x\in\R^N$ and all $t\geq 0$ we have
 $$
 0\leq u^\ep(t+t^\ep,x)\leq 1+\ep\,.
 $$
\end{itemize}
Moreover if Assumption \ref{H2} is also satisfied then there
exists a constant $\widehat {K}>1$ such that
 \begin{equation}\label{super-sol}
 u^\ep (t,x)\leq \widehat{K} U^*\left( \frac{d^*(0,x)-c^*
 t}{\ep}\right)\,,
\end{equation}
for all $t\geq 0$ and all $x\in\R^N$.
\end{thm}

In order to prove Theorem \ref{th:gene}, let us define the map
\begin{equation*}
\underline{u}(t,x):=\max\left\{0, w\left(\frac{t}{\ep}\,,
g(x)-Kt\right)\right\},
\end{equation*}
where $w(s,\xi)$ is the solution of the ordinary differential
equation \eqref{ode} and where $K>0$ is some constant to be
specified below. Then one will show the following result.
\begin{lem}[Sub-solutions for the generation]\label{LE2.1}
Let Assumption \ref{H1} be satisfied. Then for all $a>0$, there
exists $K>0$ such that, for all $\ep >0$ small enough, we have
\begin{equation*}
\underline{u}(t,x)\leq u^\ep(t,x)\,,\;\;\forall t\in [0,a \ep
|\ln\ep|]\,,\; \forall x\in \mathbb R^N\,.
\end{equation*}
\end{lem}

\begin{proof}
Let us first notice that
\begin{equation*}
\underline{u}(0,x)=g(x)\leq u^\ep (0,x)\,,\;\;\forall x\in\R^N\,.
\end{equation*}
Then we shall show that $\underline{u}$ is a sub-solution of
Problem $\Pe$. Note that if $x\notin \Omega_0$ then $g(x)=0$ and
$\underline{u}(t,x)=0$. Let us consider the operator
\begin{equation*}
\mathcal L^\ep [v]:=\partial _t v-\ep\Delta
v-\frac{1}{\ep}v(1-v)\,.
\end{equation*}
Let $a>0$ be arbitrary. We show below that, if $K>0$ is
sufficiently large then, for all $\ep >0$ small enough, $\mathcal
L^\ep [\underline{u}]\leq 0$ in the support of $\underline{u}$. In
this support we have
$$
\begin{array}{ll}
\partial_t \underline{u}=\frac{1}{\ep} w_s-K w_\xi\vsp\\
\Delta \underline{u}= w_{\xi\xi} |\nabla g|^2+w_\xi \Delta g\,.
\end{array}
$$
Then we get
\begin{equation*}
\begin{array}{lll}
\mathcal L^\ep [\underline{u}]&=\di \frac{1}{\ep} w_s-K
w_\xi-\ep\left( w_{\xi\xi} |\nabla g|^2+w_\xi
\Delta g\right)-\frac{1}{\ep} w(1-w)\vsp\\
&\leq\di \frac 1 \ep w_s -K
w_\xi-\ep\left( w_{\xi\xi} |\nabla g|^2+w_\xi \Delta g\right)-\frac 1 \ep \bar f _\ep (w)\vsp\\
&=-w_\xi\left [K+\ep\left( \frac{w_{\xi\xi}}{w_\xi} |\nabla g|^2+
\Delta g\right)\right]\,,
\end{array}
\end{equation*}
where we have successively used \eqref{fep} and \eqref{ode}. In
view of Lemma \ref{LE2} $(v)$, there exists a constant $C(a)>0$
such that, for all $(t,x)$ in the support of $\underline{u}$ with
$0\leq t \leq a\ep|\ln\ep|$, we have
\begin{equation*}
\left\lvert\frac{w_{\xi\xi}}{w_\xi} |\nabla g|^2+ \Delta
g\right\rvert \leq \frac{C(a)}{\ep}\,.
\end{equation*}
Therefore, choosing $K>C(a)$ implies
\begin{equation*}
\mathcal L^\ep [\underline{u}]\leq -w_\xi \left(K-C(a)\right)\leq
0\,,
\end{equation*}
since $w_\xi >0$.

This completes the proof of Lemma \ref{LE2.1}.
\end{proof}

Next we prove the following result.
\begin{lem}[Super-solutions]\label{LE2.2}
Let Assumptions \ref{H1} and \ref{H2} be satisfied. Then there
exists a constant $K_0>0$ such that, for all $\widehat{K}\geq
K_0$, the following holds. For all $x_0\in\Gamma
_0=\partial\Omega_0$,
 for all $\ep >0$ small enough, we have
\begin{equation*}
u^\ep (t,x)\leq \widehat{K} U^*\left(\frac{(x-x_0)\cdot
n_0-c^*t}{\ep}\right)\quad \text{ for all } (t,x) \in [0,\infty)
\times \R ^N\,,
\end{equation*}
wherein $n_0$ is the outward normal vector to $\Gamma
_0=\partial\Omega_0$ at $x_0$.
\end{lem}

\begin{proof} We recall that $\lambda$ and $M$ were defined in
Assumption \ref{H1} $(iii)$ and that, due to (\ref{behaviour}),
there exist two constants $0<\gamma^-<\gamma^-$ such that
\begin{equation*}
\gamma^- ze^{- z}\leq U^*(z)\leq \gamma^+ ze^{-z}\,,\;\;\forall
z\geq 1\,.
\end{equation*}
Thus since $\lambda\geq 1$, there exists some constant $m^->0$
such that
\begin{equation*}
U^*(z)\geq m^- e^{-\lambda z}\,,\;\;\forall z\geq 0\,.
\end{equation*}
Then we define
\begin{equation*}
K_0:= \max\left(1, \frac{M}{m^-},\;
\frac{\|g\|_\infty+M}{U^*(0)}\right)\,.
\end{equation*}

Next, let $\widehat{K}\geq K_0$ and $x_0\in\Gamma
_0=\partial\Omega_0$ be given. We consider the map
\begin{equation*}
u_* ^+(t,x):=\widehat{K} U^*\left(\frac{(x-x_0)\cdot
n_0-c^*t}{\ep}\right)\,.
\end{equation*}
Straightforward computations yield
\begin{equation*}
\ep \mathcal L^\ep [u_* ^+]=\widehat{K}(\widehat{K}-1)U^{*2}\,,
\end{equation*}
and therefore $\mathcal L^\ep [u_* ^+]\geq 0$ in $\mathbb
(0,\infty)\times \R ^N$. Hence, by the comparison principle, to
complete the proof of the lemma it is enough to prove
\begin{equation}\label{but}
u_{0,\ep}(x)\leq u_* ^+(0,x)=\widehat KU^*\left(\frac{(x-x_0)\cdot
n_0}{\ep}\right)\,,\;\;\forall x\in\R^N\,.
\end{equation}

First, assume that $x$ in the half plane $\{y\in\R^N:
\;\;(y-x_0)\cdot n_0\leq 0\}$. Then since $U^*$ is decreasing we
have
\begin{equation*}
U^*\left(\frac{(x-x_0)\cdot n_0}{\ep}\right)\geq U^*(0)\,.
\end{equation*}
Therefore we obtain that
\begin{equation*}
u_{0,\ep}(x)\leq \frac{\|g\|_\infty+M}{U^*(0)}
U^*\left(\frac{(x-x_0)\cdot n_0}{\ep}\right)\,,
\end{equation*}
which, thanks to the choice of $K_0$, yields $u_{0,\ep}(x)\leq u_*
^+(0,x)$.

Now, we assume that $x$ is in the half plane $\{y\in\R^N:
\;\;(y-x_0)\cdot n_0> 0\}$. Since $\Omega_0$ is convex, we have
$x\notin\Omega_0$. Thus $g(x)=0$ and
$$
u_{0,\ep}(x)=h_\ep(x)\leq M e^{-\lambda\frac{\|x\|}{\ep}}\,.
$$
Moreover since $0\in\Omega_0$ we have, for $x$ such that
$(x-x_0)\cdot n_0>0$,
\begin{equation*}
(x-x_0)\cdot n_0\leq \|x\|\,.
\end{equation*}
Finally since $U^*$ is decreasing, we obtain
\begin{equation*}
m^-e^{-\lambda \frac{\|x\|}{\ep}}\leq
U^*\left(\frac{\|x\|}{\ep}\right)\leq U^*\left(\frac{(x-x_0)\cdot
n_0}{\ep}\right)\,.
\end{equation*}
Therefore we get
\begin{equation*}
u_{0,\ep}(x)\leq \frac{M}{m^-}U^*\left(\frac{(x-x_0)\cdot
n_0}{\ep}\right)\,,
\end{equation*}
which, thanks to the choice of $K_0$, yields $u_{0,\ep}(x)\leq u_*
^+(0,x)$.

This completes the proof of Lemma \ref{LE2.2}.
\end{proof}

\vsp We now complete the proof of Theorem \ref{th:gene}.
\begin{proof}
The proof of $(i)$ directly follows from Lemma \ref{LE2} $(v)$
together with Lemma \ref{LE2.1}. In order to prove $(ii)$ let us
notice that the map
\begin{equation*}
\overline{u}(t,x):=w\left(\frac{t}{\ep},
\|g\|_\infty+M\right)\,,\;\;t\geq 0\,,
\end{equation*}
satisfies
\begin{equation*}
\begin{split}
&u^\ep (0,x)\leq \|g\|_\infty+M= \overline{u} (0,x) \quad\text{ for all } x\in \R ^N\,,\\
&\mathcal L^\ep [\overline u]=0 \quad\text{ in }\mathbb
(0,\infty)\times \R ^N\,.
\end{split}
\end{equation*}
Thus we obtain that
\begin{equation*}
u^\ep (t,x)\leq w\left(\frac{t}{\ep},
\|g\|_\infty+M\right)\,,\;\;\forall t\geq 0\,,\;\forall
x\in\R^N\,.
\end{equation*}
Thus Lemma \ref{LE2} $(v)$ applies and completes the proof of
$(ii)$.

Finally, under the additional Assumption \ref{H2}, the last point
of Theorem \ref{th:gene} follows from Lemma \ref{LE2.2}. Indeed
from this lemma we know that there exists $\widehat{K}>1$ such
that, for each $x_0\in\partial\Omega_0$, we have
\begin{equation*}
u^\ep(t,x)\leq \widehat{K}U^*\left(\frac{(x-x_0)\cdot
n_0-c^*t}{\ep}\right)\quad \text{ for all $t\geq 0$, all
$x\in\R^N$}\,.
\end{equation*}
Let $x\in\R^N$ be given and choose $x_0\in\partial \Omega_0$ as
the projection of $x$ on the convex $\Omega_0$. For such a choice
we have
\begin{equation*}
(x-x_0)\cdot n_0=d^*(0,x)\,,
\end{equation*}
and the result follows.

This completes the proof of Theorem \ref{th:gene}.
\end{proof}

\section{Motion of interface}\label{s:motion}

We have proved in the previous section that, as $\ep\to 0$, the
solution $\ue$ develops, after a very short time $t^\ep=\mathcal O
(\ep|\ln \ep|)$, steep transition layers that separate the region
where $\{\ue\approx 0\}$ from the one where $\{\ue\approx 1\}$. It
is the goal of this section to study the motion of interface that
then occurs in a much slower time range. Since Lemma \ref{LE2.2}
controls $\ue$ from above for all $t\geq 0$, this will be enough
to construct sub-solutions.

\vsp For all $c\in(0,c^*)$, we define $U$ as in \eqref{tw-sign}
and $V$ by
$$
V(z):=\begin{cases} \, U(z)  \quad &\textrm{ if } z< 0\\
\, 0\quad  \quad &\textrm{ if }  z\geq 0\,.
\end{cases}
$$
Next, we put
\begin{equation}\label{sub-sol}
u^- _c(t,x):=(1-\ep)V\left(\frac{d(t,x)+\ep|\ln \ep|m_1e^{m_2
t}}\ep\right),
\end{equation}
where $d$ denotes the cut-off signed distance function to the
solution of the free boundary problem $(P^c)$, as defined in
subsection \ref{ss:distance}.

\begin{lem}[Ordering initial data]\label{condition-initiale} Let Assumptions \ref{H1} and \ref{H3} be satisfied. Then there exists
$\tilde m _1 >0$ such that for all $c\in(0,c^*)$, all $m_1\geq
\tilde m _1$, all $m_2>0$, all $\ep>0$ small enough, we have
\begin{equation}
u^-_c(0,x)\leq\ue(t^\ep,x)\,,
\end{equation}
for all $x\in \R ^N$.
\end{lem}

\begin{proof}What we have to show is
\begin{equation}\label{goal}
u^-_c(0,x)=(1-\ep)V\left(\frac{d(0,x)+m_1\ep|\ln
\ep|}\ep\right)\leq \ue(t^\ep,x)\,.
\end{equation}
If $x$ is such that $d(0,x)\geq -m_1\ep |\ln\ep|$ then this is
obvious since the definition of $V$ implies $u^-_c(0,x)=0$. Next
assume that $x$ is such that $d(0,x)< -m_1\ep |\ln\ep|$. Note
that, in view of hypothesis \eqref{pente}, the mean value theorem
provides the existence of a constant $\tilde m _1>0$ such that
\begin{equation}
\text{ if } \quad d(0,x)\leq -\tilde m _1 \ep|\ln\ep| \quad \text{
then } \quad g(x)\geq k \ep|\ln\ep|\,,
\end{equation}
where $k$ is as in Theorem \ref{th:gene}. If we choose $m_1 \geq
\tilde m _1$ and $m_2>0$, then inequality \eqref{goal} follows
from Theorem \ref{th:gene} $(i)$ and the fact that $V\leq 1$.
\end{proof}

\begin{lem}[Sub-solutions for the motion]\label{super-sub-motion}
Choose $\eta >0$ such that the constants $C>0$, $\mu>0$, $d_0>0$,
$N>0$, that appear in \eqref{estimate-tw}, \eqref{norme-un},
\eqref{MVT}, \eqref{est-dist} are independent of
$c\in[c^*-\eta,c^*)$.

Then there exists $\tilde m _2 >0$ such that for all
$c\in[c^*-\eta,c^*)$, all $m_1\geq \tilde m _1$, all $m_2\geq
\tilde m _2$, all $\ep>0$ small enough, we have
\begin{equation}
\mathcal L ^\ep [u^- _c]:=\partial _t u^-_c-\ep \Delta u^-_c -\eun
u^-_c(1-u^-_c)\leq 0 \quad \text{ in } (0,\infty)\times \R^N\,.
\end{equation}
\end{lem}

\begin{proof} In the set $\{(t,x):\;\;d(t,x)\geq -\ep|\ln\ep|m_ 1 e^{m_2t}\}$ this is
obvious since $u^-_c(t,x)=0$.

We now work in the set $\{(t,x):\;\;d(t,x)< -\ep|\ln\ep|m_ 1
e^{m_2t}\}$. By using straightforward computations we get
$$
\begin{array}{lll}
\partial _t u^-_c=
(1-\ep)\left(\displaystyle{\frac{\partial _t d}{\ep}}+m_2\pt\right){U}'(\theta)\vsp \\
\Delta u^- _c= \displaystyle{\frac{|\nabla d|^2}{\ep ^2}}(1-\ep)
U''(\theta)+ \displaystyle{\frac{\Delta
d}{\ep}}(1-\ep){U}'(\theta)\vsp\\
u^-_c(1-u^-_c)=(1-\ep)U(\theta)(1-U(\theta))+\ep(1-\ep) U
^2(\theta)\,,
\end{array}
$$
where
\begin{equation}\label{theta}
\theta:=\frac{d(t,x)+\pt}\ep\,.
\end{equation}
Now, the ordinary differential equation $U''+cU'+U(1-U)=0$ yields
$$
\ep \mathcal L^\ep [u^- _c]=(1-\ep)(E_1+\cdots+E_3)\,,
$$
with
\vsp \\
$\qquad\quad  E_1:=U'(\theta)\left(\partial _t d+c+m_2\pt-\ep\Delta d\right)=:U'(\theta)F_1$\vsp \\
$\qquad\quad  E_2:=U''(\theta)(1-|\nabla d|^2)$\vsp \\
$\qquad\quad  E_3:=-\ep U^2(\theta)\,.$\vsp\\
We show below that the choice $\tilde m _2:=2N\di{(\frac 2{\tilde
m _1\mu}+1)}$ is enough to prove the lemma. To that purpose we
distinguish two cases, namely \eqref{case1} and \eqref{case2}.

\vsp First assume that
\begin{equation}\label{case1}
-\frac{m_2}{2N}\pt \leq d(t,x)<-\pt\,.
\end{equation}
It follows from \eqref{MVT} and \eqref{est-dist} that, for $\ep>0$
small enough,
$$
\begin{array}{lll}
F_1&\geq Nd+m_2\pt -\ep C\vsp \\
&\geq Nd +\frac{m_2}2\pt\vsp \\
&\geq 0\,,
\end{array}
$$
which implies that $E_1\leq 0$. Moreover, in view of
\eqref{norme-un}, we have, for $\ep>0$ small enough, $E_2=0$ and,
obviously, $E_3\leq 0$. Hence, $\mathcal L^\ep [u^-_c]\leq 0$.

\vsp Now assume that
\begin{equation}\label{case2}
d(t,x)<-\frac{m_2}{2N}\pt\,.
\end{equation}
This implies
\begin{equation}\label{theta-infini}
\theta\leq -|\ln\ep|m_1(\frac {m_2}{2N}-1)\,.
\end{equation}
Using \eqref{estimate-tw}, \eqref{MVT} and \eqref{est-dist} we see
that, for $\ep >0$ small enough,
$$
\begin{array}{lllll}
 E_1&\leq Ce^{-\mu|\theta|}(N|d|+\ep|\ln\ep|m_1m_2e^{m_2T}+\ep C)\vsp \\
 &\leq Ce^{-\mu|\theta|}(N2d_0+o(1))\vsp \\
&\leq \di{C_1 \ep ^ {m_1\mu(\frac{m_2}{2N}-1)}}\,,
 \end{array}
$$
where $C_1:=3CNd_0$. The choice of $\tilde m _2$ then forces
$E_1\leq C_1\ep ^2$. Using very similar arguments we see that
there exists $C_2>0$ such that $E_2\leq C_2\ep^2$. At least note
that $\theta \to -\infty$ as $\ep \to 0$. Hence, if $\ep>0$ is
small enough then $U(\theta)\geq b$ for some $b>0$, which in turn
implies $E_3\leq -  b^2\ep $. Collecting these estimates yields
$$
\mathcal L^\ep [u^-_c]\leq - b^2\ep+(C_1+C_2)\ep ^2 \leq 0\,, $$
for $\ep>0$ small enough.

The lemma is proved.
\end{proof}

\section{Proof of Theorem \ref{th:results}}\label{s:proof}

We are now ready to prove our main result which includes both
generation and motion of interface properties, but also provides
an $\mathcal O (\ep|\ln\ep|)$ estimate of the transition layers.
Roughly speaking, we will fit the sub- and super-solutions for the
generation into the ones for the motion.

\begin{proof} Let Assumptions \ref{H1}, \ref{H2} and \ref{H3} be
satisfied. Choose $k>0$, $\alpha >0$ and $\widehat K >1$ as in
Theorem \ref{th:gene}. As in Lemma \ref{super-sub-motion}, choose
$\eta >0$ such that the constants $C>0$, $\mu>0$, $d_0>0$, $N>0$,
that appear in \eqref{estimate-tw}, \eqref{norme-un}, \eqref{MVT},
\eqref{est-dist} are independent of $c\in[c^*-\eta,c^*)$.
According to Lemma \ref{condition-initiale}, Lemma
\ref{super-sub-motion} and the comparison principle, there exist
$m_1>0$, $m_2>0$ such that, for all $c\in[c^*-\eta,c^*]$, we have,
for $\ep >0$ small enough,
\begin{equation}\label{dessous}
u^- _c(t-t^\ep,x)\leq \ue(t,x)\,,
\end{equation}
for all $(t,x)\in[t^\ep,T]\times \R^N$. We choose $\mathcal C$
such that \begin{equation}\label{choice} \mathcal
C>\max\left(1,2(2T+m_1e^{m_2T}),\frac 2 \mu\right)\,.
\end{equation}

 Obviously, from Theorem
\ref{th:gene} $(ii)$, we have $\ue(t,x) \in[0,1+\ep]$ for all
$t^\ep \leq t\leq T$, all $x\in \R^N$.

Next we take $x\in (\R^N\setminus \overline{\Omega
^*_t})\setminus\mathcal N_{\mathcal C\ep|\ln\ep|}(\Gamma ^* _t)$,
i.e.
\begin{equation}\label{d-plus}
d^*(t,x)\geq \mathcal C \ep|\ln \ep|\,,
\end{equation}
 and prove below that $\ue(t,x)\leq  \ep$, for $t^\ep\leq t\leq
 T$. Since
 $$
 d^*(t,x)=d^*(0,x)-c^*t\,,
 $$
 we deduce from
 \eqref{super-sol}, the decrease of $U^*$ and
 \eqref{behaviour} that, for $\ep >0$ small enough, for $0\leq t\leq T$,
 $$
 \begin{array}{lll}
 \ue(t,x) &\leq \widehat K U^*(\mathcal C |\ln \ep|)\vsp\\
 &\leq \widehat K \gamma ^+ \mathcal C|\ln \ep| \ep ^{\mathcal C}\vsp \\
 &\leq \ep\,,
 \end{array}
 $$
 since $\mathcal C >1$.

At last we take $x\in\Omega ^* _ t\setminus\mathcal N_{\mathcal
C\ep|\ln\ep|}(\Gamma ^* _t)$, i.e.
\begin{equation}\label{d-moins}
d^*(t,x)\leq -\mathcal C \ep|\ln \ep|\,,
\end{equation}
 and prove below that $\ue(t,x)\geq 1-2\ep $, for $t^\ep\leq t\leq
 T$. Note that
\begin{equation}\label{lien}
 d(t-t^\ep,x)=d^*(t,x)+(c^*-c)t+ct^\ep\,.
\end{equation}
We define
\begin{equation}\label{cep}
c(\ep):=c^*- \ep|\ln\ep|\,.
\end{equation}
Combining \eqref{dessous}, with $c(\ep)$ playing the role of $c$,
and \eqref{lien}, we see that
 $$
 \ue(t,x) \geq (1-\ep)V\left(\di \frac{d^*(t,x)+ \ep|\ln\ep|t+ \ep|\ln\ep|t^\ep+\ep|\ln\ep|m_1e^{m_2
 t}}\ep\right)\,,
 $$
for $t^\ep\leq t\leq T$. In view of \eqref{d-moins}, the choice of
$\mathcal C$ in \eqref{choice} and \eqref{moins-infini}, we get,
for $\ep >0$ small enough,
$$
\begin{array}{lll}
 \ue(t,x) &\geq
(1-\ep)U\left(-\frac{\mathcal C}2|\ln\ep|\right)\vsp\\
&\geq (1-\ep)(1-Ce^{-\mu \frac{\mathcal C}2 |\ln\ep|})\vsp\\
&\geq 1-\ep-C\ep ^{\mu\frac {\mathcal C}2}\vsp\\
 &\geq 1-2\ep\,,
\end{array}
$$
since $\mathcal C >\di \frac 2 \mu$.

This completes the proof of Theorem \ref{th:results}.
\end{proof}

\section{When initial data are \lq\lq not so small"}\label{s:small}

In this section we investigate the singular limit of $\Pe$ when
initial data $u_{0,\ep}$ satisfy Assumption \ref{ass6}. We prove
below Theorem \ref{THEO6.1}.

\vsp We start with the following lemma. The proof is omitted since
it follows the same arguments as those used in Theorem
\ref{th:gene}.

\begin{lem}[Generation of interface]\label{LE6.2}
Let Assumption \ref{ass6} be satisfied. Then there exist some
constants $\alpha>0$ and $k>0$ such that, for all $\ep >0$ small
enough,
\begin{equation*}
u^\ep(t^\ep+t, x)\leq 1+\ep\,,\;\;\text{ for all } (t,x)\in
[0,\infty)\times \R^N\,,
\end{equation*}
\begin{equation*}
u^\ep(t^\ep, x)\geq 1-\ep\,,\;\;\text{ for all } x\in\R^N\,\text{
such that }\;u_{0,\ep}(x)\geq k\ep|\ln\ep|\,,
\end{equation*}
wherein $t^\ep:=\alpha \ep |\ln\ep|$.
\end{lem}

\vsp Define
\begin{equation*}
\xi_\ep:=\ep\left\{ \frac{m}{k\ep|\ln\ep|}-1\right\}^{1/n}\,.
\end{equation*}
In view of Assumption \ref{ass6}, we see that the condition
$u_{0,\ep}(x)\geq k\ep|\ln\ep|$ is satisfied when $\|x\|\leq
\xi_\ep$. Therefore, Lemma \ref{LE6.2} implies that
\begin{equation}\label{estimate}
u^\ep(t^\ep, x)\geq 1-\ep\,,\;\;\text{ for all } x\in\R^N\,\text{
such that } \;\|x\|\leq \xi_\ep\,.
\end{equation}

Now, for each $c>2$, we consider a travelling wave $(c,U)$
solution of \eqref{tw}. Then let us recall that there exist some
constants $0<m_c<M_c$ such that
\begin{equation*}
m_c e^{-\lambda_c z}\leq U(z)\leq M_ce^{-\lambda_c
z}\,,\;\;\forall z\geq 0\,,
\end{equation*}
wherein $\lambda_c>0$ is the smallest root of the equation
\begin{equation}\label{def-lambda}
\lambda^2-c\lambda +1=0\,.
\end{equation}

Then we will show the following lemma.
\begin{lem}[Sub-solutions]\label{lemme}
Let Assumption \ref{ass6} be satisfied. Let $c>2$ be given. Then
there exist some constants $\widetilde{M}>0$, $\widetilde{k}>0$
and $\widehat{k}>0$ such that
\begin{equation*}
u^\ep (t^\ep,x)\geq  \widetilde{M} \exp \left\{-\lambda_c
\frac{\|x\|-\widetilde{k}\ep|\ln\ep|}{\ep}\right\}\,,\;\;\text{ if
}\; \|x\|\geq \widehat{k}\ep|\ln\ep|\,.
\end{equation*}
\end{lem}

\begin{proof}
Let $(c,U)$ be a given travelling wave associated with the given
speed $c>2$. Let $c_1\in(0,c)$ be given and fixed. Let $\rho>0$ be
chosen large enough such that
\begin{equation*}
\begin{split}
&\rho\geq \max\left(\frac{N-1}{c-c_1},\frac n {\lambda _c}\right)\vsp\\
&\frac{m}{1+\rho^n}\geq M_c e^{-\lambda_c\rho}\,.
\end{split}
\end{equation*}
Then we consider the map $v_0=v_0(s)$ defined by
\begin{equation*}
v_0(s):=\begin{cases} U(\rho)&\quad\text{ if }\,|s|\leq\rho\\
U(|s|)&\quad\text{ if }\,|s|\geq \rho\,,\end{cases}
\end{equation*}
and define the function $W(t,x)$ by
\begin{equation*}
W(t,x):=v_0\left(\frac{\|x\|-c_1t}{\ep}\right).
\end{equation*}

Note that due to the choice of $\rho$ and the definition of $v_0$
we have
\begin{equation*}
\frac{m}{1+\frac{\|x\|^n}{\ep^n}}\geq W(0,x)\,,\;\;\text{ if }\;
\|x\|\leq \ep\rho\,.
\end{equation*}
Moreover, since $\rho\geq n/\lambda _c$, we see that
\begin{equation*}
\frac{m}{1+\frac{\|x\|^n}{\ep^n}}\geq M_c
e^{-\lambda_c\frac{\|x\|}{\ep}}\,,\;\;\text{ if }\; \|x\|\geq
\ep\rho\,.
\end{equation*}
Then we get that
\begin{equation*}
 u_{0,\ep}(x)\geq \frac{m}{1+\frac{\| x\| ^n}{\ep^n}}\geq W(0,x)\,,\;\;\forall x\in\R^N\,.
\end{equation*}

On the other hand straightforward computations yield
\begin{equation*}
\mathcal L^\ep [W](t,x)=\begin{cases}
 -\frac{1}{\ep}U(\rho)(1-U(\rho))&\quad\text{ if }\,\left|\|x\|-c_1t\right|\leq \ep\rho\vsp\\
 \frac{1}{\ep} U'\left(\frac{\|x\|-c_1t}{\ep}\right)
\left(c-c_1+\ep\frac{N-1}{\|x\|}\right)&\quad\text{ if
}\,\left|\|x\|-c_1t\right|\geq\ep\rho\,.\end{cases}
\end{equation*}
Thus $\mathcal L^\ep [W](t,x)\leq 0$ for all $(t,x)\in
(0,\infty)\times \R^N$.

From the comparison principle, we then deduce that
\begin{equation*}
W(t,x)\leq u^\ep(t,x)\,,
\end{equation*}
for all $(t,x)\in (0,\infty)\times \R^N$. As a consequence we
obtain that
\begin{equation*}
u^\ep (t^\ep,x)\geq m_c e^{-\lambda_c \frac{\|x\|-c_1
t^\ep}{\ep}}\,,\;\;\text{ if }\, \|x\|\geq c_1t^\ep+\ep\rho\,.
\end{equation*}
Using this, one then easily proves that the choices $\widetilde
M=m_c$, $\widetilde k=c_1 \alpha$ and $\widehat k=2c_1\alpha$ are
enough to conclude.

Lemma \ref{lemme} is proved.
\end{proof}

We are now in the position to prove Theorem \ref{THEO6.1}.

\begin{proof} Let $(t_0,x_0)\in (0,\infty)\times\R^N$ be given. Let $c>\max(\frac{\|x_0\|}{t_0},2)$ be given.
Fix $c_1\in(0,c)$ such that
\begin{equation*}
\|x_0\|-c_1 t_0<0\,.
\end{equation*}
Let $\widehat{\ep}\in (0,1)$ be given and fixed.  We choose
$\rho>0$ large enough such that
\begin{equation*}
c_1+\frac{N-1}{\rho}\leq c\,.
\end{equation*}
Let $U$ be the travelling wave solution of \eqref{tw} associated
with the wave speed $c$ such that
\begin{equation}\label{widehatep}
U(0)= 1-\widehat{\ep}\,.
\end{equation}
Next, for each $\rho >0$, we consider the map
\begin{equation*}
q (s):=\begin{cases}  U(0) &\quad\text{ if }\,s\leq \rho\\
U(s-\rho)&\quad\text{ if }\,s\geq \rho\,,
\end{cases}
\end{equation*}
and define the function $\widetilde W (t,x)$ by
\begin{equation*}
\widetilde{W}(t,x):=q\left(\frac{\|x\|-c_1t}{\ep}\right)\,.
\end{equation*}

Using similar arguments to those used in the proof of Lemma
\ref{lemme}, we see that
\begin{equation*}
\mathcal L^\ep [\widetilde{W}]\leq 0\quad\text{ in }\mathbb
(0,\infty)\times \R ^N\,.
\end{equation*}

Next, we prove below that there exists $\ep_1\in(0,\widehat{\ep})$
such that, for all $\ep\in (0,\ep_1)$,
\begin{equation}\label{qqch}
\widetilde{W}(0,x)\leq u^\ep(t^\ep,x)\,,\;\;\forall x\in\R^N\,.
\end{equation}
Indeed, it directly follows from \eqref{estimate} and
\eqref{widehatep} that \eqref{qqch} holds true if $\|x\|\leq
\xi_\ep$ and $\ep\leq \widehat{\ep}$. Let us now assume that
$\|x\|\geq \xi_\ep$. Note that, for $\ep >0$ small enough,
\begin{equation*}
\ep\rho\leq \xi_\ep\,,\;\;\widehat{k}\ep|\ln\ep|\leq \xi_\ep\,.
\end{equation*}
Therefore, we have
$$
\widetilde{W}(0,x)\leq M_c
e^{-\lambda_c\frac{\|x\|-\rho\ep}{\ep}}\,.
$$
Since, by Lemma \ref{lemme}, we have
$$
u(t^\ep,x)\geq \widetilde{M} e^{-\lambda_c
\frac{\|x\|-\widetilde{k}\ep|\ln\ep|}{\ep}}\,,
$$
it follows that \eqref{qqch} holds true as well in the case
$\|x\|\geq \xi _\ep$, for $\ep>0$ small enough.

The comparison principle then applies and yields, for $\ep >0$
small enough,
\begin{equation*}
\widetilde{W}(t-t^\ep,x)\leq u^\ep(t,x)\,,\;\;\forall (t,x)\in
[t^\ep ,\infty)\times\R^N\,.
\end{equation*}

Since $\|x_0\|-c_1t_0<0$ and $\displaystyle \lim_{\ep\to 0^+}
\rho\ep-c_1 t^\ep=0$, we see that, for $\ep >0$ small enough,
$$
\frac{\|x_0\|-c_1(t_0-t^\ep)}\ep \leq \rho\,,
$$
which in turn implies that
\begin{equation*}
 u^\ep(t_0,x_0)\geq \widetilde{W}(t_0-t^\ep,x_0)= U(0)\,.
\end{equation*}
Thus we get
\begin{equation*}
U(0)\leq \liminf_{\ep\to 0} u^\ep(t_0,x_0)\,.
\end{equation*}
Since $U(0)=1-\widehat \ep$, with $\widehat{\ep}>0$ arbitrary
small, we obtain
\begin{equation*}
1\leq \liminf_{\ep\to 0} u^\ep(t_0,x_0)\,.
\end{equation*}
Finally, due to the first part of Lemma \ref{LE6.2}, we get that
\begin{equation*}
\limsup_{\ep\to 0} u^\ep(t_0,x_0)\leq 1\,,
\end{equation*}
which completes the proof of Theorem \ref{THEO6.1}.

\end{proof}

\end{document}